\documentclass[12pt,letterpaper]{amsart}

\usepackage{amsfonts}
\usepackage{amsthm}
\usepackage{amsmath}
\usepackage{amscd}
\usepackage{amssymb}
\usepackage[latin2,utf8]{inputenc}
\usepackage{csquotes}
\usepackage{multicol}
\usepackage{t1enc}
\usepackage[mathscr]{eucal}
\usepackage{indentfirst}
\usepackage{graphicx}
\usepackage{graphics}
\usepackage{pict2e}
\usepackage{epic}
\numberwithin{equation}{section}
\usepackage{enumitem}
\usepackage{leftidx}
\usepackage[margin=4cm]{geometry}
\usepackage{epstopdf}
\usepackage{color}
\usepackage{dsfont}
\usepackage{url}

\theoremstyle{plain}
\newtheorem{Th}{Theorem}[section]
\newtheorem{Lemma}[Th]{Lemma}

 \theoremstyle{definition}
\newtheorem{Def}[Th]{Definition}

\newtheorem{Rem}[Th]{Remark}
\newtheorem{?}[Th]{Problem}

\providecommand{\vol}[1]{\left\lvert#1\right\rvert}
\providecommand{\abs}[1]{\left\lvert#1\right\rvert}

\providecommand{\ang}[1]{\left\langle#1\right\rangle}
\providecommand{\parenth}[1]{\left(#1\right)}

\newcommand{\R}{\mathbb{R}} %%
\newcommand{\E}{\mathbb{E}} %%
\newcommand{\Pro}{\mathbb{P}} %%
\newcommand{\Sp}{\mathbb{S}} %%
\setenumerate{itemsep=.3em}

\usepackage{hyperref}

\begin{document}

\setlength{\abovedisplayskip}{1em}
\setlength{\belowdisplayskip}{1em}

\title[Stoch. reverse isoperimetric inequalities in the plane]{Stochastic reverse isoperimetric inequalities in the plane}

\author[J. Rebollo Bueno]{Jesus
  Rebollo Bueno}

\address{University of Missouri \\ Department of Mathematics
  \\ Columbia MO 65201}

\email{jrc65@mail.missouri.edu}

% \subjclass[]{}

 \keywords{}

\begin{abstract}
   In recent years, it has been shown that some classical inequalities follow from a local stochastic dominance for naturally associated random polytopes. We strengthen planar isoperimetric inequalities by attaching a stochastic model to some classical inequalities, such as Mahler's Theorem, and a reverse Lutwak-Zhang inequality, the polar for $L_p$ centroid bodies. In particular, we obtain the dual counterpart to a result of Bisztriczky-B\"or\"oczky.
\end{abstract}

\maketitle

\section{Introduction}
\label{Intro}

In this paper we study functionals of convex bodies invariant under the general linear group, $GL(n)$. To pave the way, we begin with a fundamental example. Let $K$ be a convex body in $\R^n$, i.e., a compact convex set with non-empty interior. Whenever $K$ is symmetric with respect to the origin the volume product is defined by
    \begin{equation}
    \label{Volume-Product}
        P(K)
        :=
        \vol{K}\vol{K^{\circ}},
    \end{equation}
where $\vol{\cdot}$ denotes the $n$-dimensional Lebesgue measure and $K^{\circ}$ is the polar body of $K$ (see \S\ref{Prelim} for precise definitions). Since $P$ is continuous and $GL(n)$ invariant, a compactness argument shows that it attains a maximum and a minimum. The upper estimate is known as the Blaschke-Santal\'o inequality \cite{Bla:1923-II,San:1949,meypajblaschke}:
    \begin{equation}
    \label{Symmetric-Blaschke-Santalo}
        P(K)
        \leq
        P(B_2^n),
    \end{equation}
where $B_2^n$ denotes the $n$-dimensional Euclidean unit ball centered at the origin. In addition, Alexander, Fradelizi, and Zvavitch \cite{AleFraZva:2019-Polytope} recently showed that among polytopes simplicial ones are maximizers; see \cite{MeyRei:2011-Polygons} for earlier work in the plane. On the other hand, the lower estimate is known as the Mahler's Conjecture, confirmed by Mahler for dimension 2 \cite{Mah:1939-Minimal}, and it is still one of the main open problems in convex geometry:
\begin{equation}\label{Mahler-Conjecture}
	P(K) \geq P(L),
\end{equation}
where $L$ denotes the $n$-parallelotope for the symmetric case, or a simplex for the non-symmetric one. A recent breakthrough by Iriyeh and Shibata \cite{IriShi:2017symmetric} solves it in dimension $3$ (see Fradelizi, Hubard, Meyer, Rold\'{a}n-Pensado, and Zvavitch \cite{FraHubMeyRol:2019Equipartitions} for a shorter proof). The $n$-dimensional statement is known in special cases. Saint Raymond established it for unconditional bodies \cite{Saint:1980} (see Meyer \cite{Mey:1986-Caracterisation} for a simpler proof). Reisner proved the result for zonoids \cite{Rei:1986-Zonoids}. For local versions and other known special cases see \cite{BarFra:2013-Volume,Kar:2019-Mahler,KimRei:2011-Simplex,LopRei:1998-Mahler,NazPetRyaZva:2010,Stan:2009-Mahler,ReiSchWer:2012-Mahler}.
Moreover, the celebrated result by Bourgain and Milman \cite{BouMil:1987} established \eqref{Mahler-Conjecture} up to a constant. See \cite{GiaPaoVri:2014-Isotropic,Naz:2012-BM,Pis:1999-Volume,Sar:2010-Characterizations} for other proofs and related results. The sharpest known constant is due to Kuperberg \cite{Kup:2008}.

On the topic of functional inequalities related to convex bodies, Paouris and Pivovarov studied stochastic forms of isoperimetric inequalities \cite{paopivprob}. Subsequently, Cordero-Erausquin, Fradelizi, Paouris, and Pivovarov \cite{CFPP} showed randomized inequalities for polar bodies. A typical example of such random sets is given by the convex hull of the columns of a random matrix, for which the expectation of the volume is maximized by $N$ independent random vectors uniformly distributed in the Euclidean ball of volume one.
In particular, motivated by the work in Stochastic Geometry \cite{Bla:1917,Bus:1953,Gro:1974,camcolgroanote,Cam-Gron:2006volume,campigronchionvolume}, they gave a stochastic Blashke-Santal\'{o} inequality. Let $K$ be a symmetric convex body in $\R^n$, $\{X_i\}_{i=1}^N$ random vectors sampled uniformly in $K$, and $\{Y_i\}_{i=1}^N$ sampled uniformly in $K^*$, the Euclidean ball of the same volume as $K$. Then
    \[
        \E\vol{([K]_N)^{\circ}} \leq \E\vol{([K^*]_N)^{\circ}},
    \]
where $[K]_N$ stands for ${\rm conv}\{\pm X_1,\dots,\pm X_N\}$ and similarly for $[K^*]_N$.
One may find the origin of this in 1864, when J.J. Sylvester \cite{Syl:1864} posed his four points problem, which ultimately led to the study of
    \begin{equation}
    \label{sylvestersfunctional}
        M(K)
        =
        \dfrac{1}{\vol{K}}\E\vol{[K]_{N}},
    \end{equation}
for $N\geq n+1$. Thus, Sylvester's problem is equivalent to study the extremals of $M(K)$.
Blaschke \cite{Bla:1917} showed that in two dimensions the minimum is attained if and only if $K$ is an ellipse, and triangles are the only maximizers of \ref{sylvestersfunctional} for $N=4$ (see \cite{DalLar:1990-Volumes} for the extension to $N$ points on the plane).
Indeed, since $M(K)$ is not increasing under Steiner symmetrization \cite{Gro:1973-Some}, ellipsoids are the only minimizers for the $n$-dimensional Sylvester's problem.
The maximum problem is still open for $n>3$ and, as mentioned before, it is conjectured that the simplices are the only maximizers.
Moreover, Campi, Colesanti, and Gronchi used RS-Movements (see \S\ref{Prelim}) to determine maximizers of the $r$-th order moment of \eqref{sylvestersfunctional} in the plane for $N>n$. See Meckes \cite{Mec:2005-Sylvester} for the symmetric case.

We are interested in the polar version of Sylvester's functional. Let $K$ be a centrally symmetric convex body in $\R^n$, $\{X_i\}_{i=1}^n$ independent random vectors sampled uniformly in $K$. We study a generalization of the functional
     \begin{equation*}
        W(K)
        =
        \dfrac{1}{\vol{K}}\E\vol{([K]_{N})^{\circ}}^{-1},
     \end{equation*}
where $N>n$. In particular, we are interested in the functional of the normalized higher order moments of the volume of the polar of a random polytope in a centrally symmetric convex body in $\R^n$. Such functional, as Sylvester's, is continuous with respect to the Hausdorff metric and invariant under invertible linear transformations.

The convexity of $W(K)$ under RS-Movements, Theorem \ref{maintheorem}, together with the setting beyond convex hull from Paouris and Pivovarov \cite{paopivprob}, provides a path to prove stochastic inequalities on the plane that strengthen planar isoperimetric inequalities. In particular, for centrally symmetric convex bodies $K$ in $\R^2$ with the origin as an interior point and $N\geq2$, we are able to find a stochastic Mahler's inequality.
\begin{Th}
\label{Stochastic-Symmetric-Mahler}
	Let $N\geq2$, $r\geq1$, and $K$ a centrally symmetric convex body in $\R^2$. Then
        \begin{equation}
        \label{Stochastic-Mahler}
            \E\vol{([K]_N)^{\circ}}^{-r} \leq \E\vol{([Q]_N)^{\circ}}^{-r},
        \end{equation}
     where $Q$ is a square with $\vol{K}=\vol{Q}$.
\end{Th}
We will also be able to derive a stochastic reverse Lutwak-Zhang's inequality on the plane. The deterministic inequality on the plane was shown by Campi and Gronchi \cite{campigronchionvolume}. Let $K$ be a star body about the origin in $\R^n$ and $1\leq p\leq\infty$. The $p$-centroid body of $K$, $Z_pK$, is defined to be the convex body given by the function
    \begin{equation}
    \label{polar-p-centroid-body}
      \displaystyle h(Z_pK,x)
      =
      \left(
      \dfrac{1}{\vol{K}}\int_K \abs{\ang{x,z}}^p dz
      \right)^{1/p}.
    \end{equation}
In 1997 Lutwak and Zhang \cite{LutZha-1997-Blaschke} showed, with a different normalization, that
    \begin{equation}
    \label{Lutwak-Zhang}
        \vol{(Z_pK)^{\circ}} \leq \vol{(Z_pK^*)^{\circ}}.
    \end{equation}
Thus, one can think of the latter as a generalization of the Blaschke-Santal\'{o} inequality \eqref{Symmetric-Blaschke-Santalo}, since it can be obtained as the limit case $p=\infty$ for \eqref{Lutwak-Zhang}. For more recent developments see \cite{LutYanZha:2000-Lp,LutYanZha:2010-Orlicz}.

In order to give a stochastic form of \eqref{Lutwak-Zhang} on the plane, let $K$ be a centrally symmetric convex body in $\R^2$, $N\geq 2$, $\{X_i\}_{i=1}^N$ independent random vectors uniformly distributed in $K$, $r\geq1$, and $1\leq p\leq\infty$. We define the empirical $p$-centroid body of $K$, $Z_{p,N}(K)$, by its (random) support function
    \begin{equation}
    \label{Empirical-p-Centroid-Body}
        h^p(Z_{p,N}(K),z)
        =
        \frac{1}{N}\displaystyle\sum\limits_{i=1}^N\abs{\ang{X_i,z}}^p.
    \end{equation}
We denote $Z^{\circ}_{p,N}(K)=(Z_{p,N}(K))^{\circ}$, and $Z_NK$ the empirical centroid body of $K$, case $p=1$.

\begin{Th}
\label{Stochastic-Reverse-Lutwak-Zhang}
    Let $K$ be a centrally symmetric convex body in $\R^2$. Then
        \begin{equation}
        \label{Eq-Stochastic-Reverse-Lutwak-Zhang}
        \E\vol{Z^{\circ}_{p,N}(K)}^{-r}
        \leq
        \E\vol{Z^{\circ}_{p,N}(Q)}^{-r},
        \end{equation}
    where $\vol{K}=\vol{Q}$. In particular, when $N\to\infty$
        \begin{equation}
        \label{Reverse-Lutwak-Zhang}
        \vol{(Z_pK)^{\circ}} \geq \vol{(Z_pQ)^{\circ}}.
        \end{equation}
\end{Th}

Bisztriczky and B\"{o}r\"{o}czky \cite{BisBor-2001-Centroid} provided a planar converse of the Busemann-Petty centroid inequality \cite{Bus:1953,BusStr:1960-Area,Pet:1961-Centroid,Pet:1971-Isoperimetric}. We recall the latter states that given an origin symmetric convex body $K\subset\R^n$ then
    \[
        \vol{ZK} \geq \vol{ZK^*}.
    \]
where $ZK$ denotes the centroid body, i.e., the case $p=1$. Bisztriczky and B\"{o}r\"{o}czky showed that given a centrally symmetric convex body $K\subset\R^2$, with $\vol{K}=\vol{Q}$, then
    \[
        \vol{ZK} \leq \vol{ZQ}.
    \]
Theorem~\ref{Stochastic-Reverse-Lutwak-Zhang} for $p=1$ gives a stochastic polar version of such inequality, that is
        \begin{equation}
        \label{Eq-Stochastic-Polar-Bis-Bor}
            \mathbb{E}\vol{Z^{\circ}_{N}(K)}^{-r} \leq \vol{Z^{\circ}_{N}(Q)}^{-r}.
        \end{equation}
When $N\to\infty$ and $r=1$, one obtains the mentioned deterministic inequality
        \begin{equation}
        \label{Polar-Bis-Bor}
            \vol{(ZK)^{\circ}}
            \geq
            \vol{(Z Q)^{\circ}}.
        \end{equation}
Furthermore, we will be able to generalize Theorem \ref{maintheorem} for non-symmetric convex bodies $K$ in $\R^n$ considering a generalization of the functional
    \begin{equation*}
        W^{s\circ}(K)
        =
        \dfrac{1}{\vol{K}}\E\vol{([K]_{N})^{s\circ}}^{-1},
    \end{equation*}
where $N\geq n+1$, $[K]_{N}=\mathop{\rm conv}\{  X_1,\ldots, X_{N}\}$, and $K^{s\circ}$ denotes the polar body of $K$ with respect to its Santal\'o point. This will allow us to prove a stochastic non-symmetric Mahler's inequality.

\begin{Th}
\label{Stochastic-Non-Symmetric-Mahler}
	Let $K$ be a convex body in $\R^2$, $r\geq1$, and $N\geq3$. Then
        \begin{equation}
        \label{Stochastic-Mahler-Santalo}
            \E\vol{([K]_N)^{s\circ}}^{-r} \leq \E\vol{([T]_N)^{s\circ}}^{-r},
        \end{equation}
where $T$ denotes a triangle with centroid at the origin and $\vol{T}=\vol{K}$.
\end{Th}

We make special use of the work by Rogers and Shephard \cite{rogsheext,Shephard:1964shadow} and developed by Campi, Colesanti, and Gronchi \cite{camcolgroanote,Cam-Gron:2002reverse,campigronchionvolume,Cam-Gron:2006volume}. For more about shadow systems see Saroglou \cite{Sar:2013-shadow}.

\section{Preliminaries}
\label{Prelim}

A set is (centrally) symmetric if $K=-K$. Let $K$, $L$ be two sets in $\R^n$, their Minkowski sum is given by
    \[
    K+L = \{ k+l : k\in K, l\in L\},
    \]
and the Hausdorff distance between $K$ and $L$ by
    \[
    \delta^H(K,L) = \inf\{\epsilon>0 : K\subset L+\epsilon B_2^n, L\subset K+\epsilon B_2^n\}.
    \]

We call a set convex if for all $\lambda\in(0,1)$ and $x,y\in\R^n$, then $\lambda{x}+(1-\lambda)y\in{K}$. Let $K$ be a convex set, $\theta\in\Sp^{n-1}$, and
$P_{\theta^{\perp}}$ the orthogonal projection onto
$\theta^{\perp}$. We define the upper function of $K$ with respect to
$\theta$ as
\begin{eqnarray*}
  \label{upperfunction}
  u_{K,\theta}:P_{\theta^{\perp}}K&\longmapsto&\R\\ u_{K,\theta}(y)&=&
  \sup\{\lambda:y+\lambda\theta\in K\}.
\end{eqnarray*}
Analogously we define the lower function of $K$ with respect to
$\theta$ as
\begin{align*}
  \label{lowerfunction}
  \ell_{K,\theta}:&P_{\theta^{\perp}}K\longmapsto\R\\
  &\ell_{K,\theta}(y)
  =\inf\{\lambda:y+\lambda\theta\in K\}.
\end{align*}
Whenever $K$ and $\theta$ are clear by the context we will just write $P$, $u$, and $\ell$. Also notice that, for $K$ convex, $u$ and $\ell$ are concave and convex, respectively.

%We recall, the Steiner Symmetral of a non-empty compact set
%$A\subset\R^n$ with respect to $\theta^{\perp}$,
%$S_{\theta^{\perp}}A$, is the set such that for each line
%$G$ orthogonal to $\theta^{\perp}$ and meeting $A$, the set $G\cap
%S_{\theta^{\perp}}A$ is a closed segment with midpoint on
%$\theta^{\perp}$, and length equal to that of the set $G\cap A$. The
%mapping $S_{\theta^{\perp}}:A\mapsto S_{\theta^{\perp}}A$ is called
%the Steiner Symmetrization of $A$ with respect to $\theta^{\perp}$. In
%particular, if $K$ is a convex body
%\begin{equation*}
%	S_{\theta^{\perp}}K = \{x+\lambda z:x\in PK, -\dfrac{u(x)-\ell(x)}{2} \leq \lambda \leq \dfrac{u(x)+\ell(x)}{2}\}.
%\end{equation*}
%This shows that $S_{\theta^{\perp}}K$ is convex, since the function $u-\ell$ is concave. Moreover, $S_{\theta^{\perp}}K$ is symmetric with respect to $\theta^{\perp}$, closed, and by Fubini's theorem it has the same volume as $K$.

Let $\langle \cdot,\cdot \rangle$ be the usual inner product, the support function of a convex set is defined as
    \[
    h_K(w) = \underset{x\in K}{\max} \langle w,x \rangle, \forall w\in\R^n.
    \]
One can define the $p$-centroid body of $K$ by its support function as
    \begin{equation}
    \label{p-centroid-body}
        \displaystyle h_{Z_pK}(z) = \left( \dfrac{1}{\vol{K}}\int_K \abs{\langle x,z \rangle}^p dx\right)^{1/p}.
    \end{equation}
Thus, the volume of $Z_pK$ and $Z_p\Phi K$ is the same whenever $\Phi$ is a linear transformation with determinant one. Indeed, the centroid body itself is an affine equivariant, see \cite{LutYanZha:2000-Lp}.

The polar set, $K^{\circ}$, of $K$ is given by
\begin{equation}\label{Polar-Body}
	K^{\circ} = \{\omega \in\R^n : \langle \omega,x \rangle \leq 1, \forall x\in K\}.
\end{equation}
Notice that $K^{\circ}$ depends on the location of the origin and it follows from the definition that if $K$ contains the origin $(K^{\circ})^{\circ} = K$. In addition, a convex set is said to be a convex body if it is also compact with non-empty interior. In such case, the volume of the polar body of $K$ can be determined using the support function by
    \begin{equation}\label{polarvolume}
        \vol{K^{\circ}} = \dfrac{1}{n}\displaystyle\int_{\Sp^{n-1}} h_{K}^{-n}(x)\,dx,
    \end{equation}
where $\Sp^{n-1}$ denotes the unit sphere in $\R^n$, and $|\cdot|$ the Lebesgue measure on $\R^n$. For the case of a non-symmetric convex body we define
    \[
    K^{s\circ} = (K-s)^{\circ} = \{\omega\in\R^n : \langle \omega,x-s \rangle \leq 1, \forall x\in K\},
    \]
where $s$ denotes the Santal\'o point of $K$, i.e. the unique point in the interior of $K$ such that
    \[
    \vol{K^{s\circ}} = \underset{x\in\mathop{\rm int(K)}}{\min}\{\vol{(K-x)^{\circ}}\}.
    \]

For $\{x_1,\ldots,x_N\}\subset\R^n$ we denote by $[\mathbf{x}]:=[x_1\cdots x_N]$ the linear operator from $\R^N$ to $\R^n$, and $[\mathbf{x}]C$ the set
	\[
	[\mathbf{x}]C = \left\{ \displaystyle\underset{i=1}{\overset{N}{\sum}} c_ix_i : c=(c_i)\in C \right\} \subset \R^n.
	\]
An example of this is when we consider $C=B_1^N$, then
    \[
    [\mathbf{x}]B_1^N = \mathop{\rm conv}\{ \pm x_1,\ldots,\pm x_N\},
    \]
where $\mathop{\rm conv}$ stands for convex hull.

\section{Dual Version of Sylvester's Functionals}

We begin by recalling the notion of shadow system introduced by Rogers and Shephard \cite{rogsheext}.
\begin{Def}
    A shadow system along a direction $\theta\in \Sp^{n-1}$ is a family of convex sets $K_t\subset\R^n$ defined by
    \[
    K_t = \mathop{\rm conv}\{x+t\alpha(x)\theta : x\in A\subset \R^n\},
    \]
    where $A$ is an arbitrary bounded set of points, $\alpha$ is a bounded function on $A$, and $t$ belongs to an interval of the real axis.
\end{Def}
\noindent Rogers and Shephard \cite{rogsheext,Shephard:1964shadow} proved that the volume of $K_t$ is a convex function of $t$, and many isoperimetric type inequalities have been shown using this technique. In the dual case, Campi and Gronchi showed the following fundamental theorem \cite{campigronchionvolume}:
\begin{Th}\label{camgroth}
    If $K_t$, $t\in[0,1]$, is a shadow system of origin symmetric convex bodies in $\R^n$, then $|K_t^{\circ}|^{-1}$ is a convex function of $t$.
\end{Th}
\noindent We are interested in a particular case of shadow systems studied by Campi, Colesanti, and Gronchi \cite{camcolgroanote}, where the bounded function $\alpha$ is constant on each chord of $K$ parallel to $\theta$.
\begin{Def}
    Let $K\subset\R^n$. A shadow system is called an RS-movement of $K$ if
        \[
        K_t = \mathop{\rm conv}\{ x+ t\beta(Px)\theta : x\in K\},
        \]
    where $t\in[a,b]$, $0\in[a,b]$, $P$ is the orthogonal projection defined as in \S~\ref{Prelim}, and $\beta$ is a real valued function on $PK$.
\end{Def}

Now, we introduce a functional motivated by Sylvester's. It expresses the normalized higher negative moments of the volume of the polar of a random polytope in $K$. Denote by $\mathcal{K}^n$ the class of all convex bodies in $\R^n$ and $\mathcal{K}_{0}^n$ the class of all convex bodies with the origin as an interior point. Let $K \in\mathcal{K}_{0}^n$ and $C \in\mathcal{K}_{0}^N$ be centrally symmetric, $r\geq1$, $N\geq n$, and $X_1,\dots,X_N$ independent random vectors sampled uniformly in $K$. We define the functional
\begin{align}\label{mainfunctional}
    W_r(K;N;C) 
    &=\nonumber
    \frac{1}{|K|^{N+r}}\displaystyle\int_{K^N} |([x_1\cdots x_N]C)^{\circ}|^{-r} \, dx_1\cdots dx_N\\[0.5em]
    &=
    \dfrac{1}{\vol{K}^{r}}\E\vol{\left([X_1\cdots X_N]C\right)^\circ}^{-r}.
\end{align}
Notice that this functional is finite. Indeed, for $[K]_N=[x_1\cdots x_N]B_1^N$, $[K]_N\subset K$ implies $\vol{([K]_N)^{\circ}}\geq\vol{K^{\circ}}$. Thus, $\vol{([K]_N)^{\circ}}^{-r}\leq\vol{K^{\circ}}^{-r}$ for all $r>0$. For general $C$, we have $C\subset RB_1^N$ for some $R>0$. Hence, $W_r(K;N,C)$ is finite. In addition, it is easy to check that $W_r(K;N;C)$ is continuous with respect to the Hausdorff metric and it is invariant under invertible linear transformations $T$, i.e., $W_r(TK;N,C)=W_r(K;N,C)$. This follows from the change of variabes $x_i=Ty_i$ in \eqref{mainfunctional} and using $([Ty_1\ldots Ty_N]C)^{\circ}=T^{-t}([y_1\ldots y_N]C)^{\circ}$ However, in general, the functional is not invariant under invertible affine transformations as we are considering the polar body with respect to the origin.
	\begin{Th}\label{maintheorem}
    Let $C\in\mathcal{K}_{0}^N$ be centrally symmetric, $K_t$ an RS-Movement for $t\in[-1,1]$, $r\geq 1$, and $N\geq n$. Then
    	\[
    	t\longmapsto W_r(K_t;N;C)
    	\]
    is a convex function of $t$.
	\end{Th}
We will use an analogue of Theorem~\ref{camgroth} for linear images of convex sets, see \cite[Corollary 3.8]{CFPP}. Namely, given an origin-symmetric convex body  $C$ in $\R^N$, $\theta\in\Sp^{n-1}$, and $x_1,\ldots,x_N\in\theta^{\perp}$, the map
	\begin{equation}\label{CFPP-Conv}
	(t_1,\ldots,t_N) \mapsto \vol{([x_1+t_1\theta \cdots x_N+t_N\theta]C)^{\circ}}^{-1}
	\end{equation}
is convex on $\R^N$.
\begin{proof}
    Suppose without loss of generality $K_0=K$, so $PK=PK_t$ and $\vol{K}=\vol{K_t}$ for all $t\in[-1,1]$. Let $u$ and $\ell$ be as in \S~\ref{Prelim}, then one has
	\begin{align*}
		K &= \{ (x,y)\in PK\times\R: \ell(x)\leq y \leq u(x)\}\\[.5em]
		K_t &= \{(x,y)\in PK\times\R : (\ell+t\beta)(x) \leq y \leq (u+t\beta)(x)\}.
	\end{align*}
    Therefore by Fubini's Theorem
\begin{align*}
        W_r(K_t;N;C)
        &=
        \dfrac{1}{\vol{K}^{N+r}}
        \displaystyle\int_{(PK)^N}
        \left(
        \prod\limits_{i=1}^N\int_{(\ell+t\beta)(x_i)}^{(u+t\beta)(x_i)}
        \vol{
        \left(
        M_1C
        \right)^\circ
        }^{-r}
        d\vec{y}
        \right)
        d\vec{x}\\[1em]
        &=
            \dfrac{1}{\vol{K}^{N+r}}
            \displaystyle\int_{(PK)^N}
            \left(
            \prod\limits_{i=1}^N\int_{\ell(x_i)}^{u(x_i)}
            \vol{
            (
            M_2C
            )^\circ
            }^{-r}
            d\vec{\widetilde{y}}
            \right)
            d\vec{x},
\end{align*}
    where
        \begin{subequations}
            \begin{align*}
                M_1C
                &:=
                [x_1+y_1\theta\cdots x_N+y_N\theta]C\\
                &=
                \left\{ \displaystyle\underset{i=1}{\overset{N}{\sum}} c_i(x_i+y_i\theta) : c=(c_i)\in C \right\}\\[0.5em]
    \intertext{and after considering $y_i=\widetilde{y}_i+t\beta(x_i)$}
                M_2C
                &:=
                [x_1+(\widetilde{y}_1+t\beta(x_1))\theta\cdots x_N+(\widetilde{y}_N+t\beta(x_N))\theta]C\\[0.5em]
                &=
                \left\{ \displaystyle\underset{i=1}{\overset{N}{\sum}} c_i(x_i+(\widetilde{y}_i+t\beta(x_i))\theta) : c=(c_i)\in C \right\}.
            \end{align*}
        \end{subequations}
    Notice that by the convexity of \eqref{CFPP-Conv}, the functional in \eqref{mainfunctional} is the repeated integral of the $r$-th power of a convex functional. Hence, it is convex.

\end{proof}

At this point we generalize the previous functional for non-symmetric convex bodies. Let $N\geq n+1$, $K \in\mathcal{K}^n$, $C\in\mathcal{K}^N$, $X_1,\dots,X_N$ independent random vectors sampled uniformly in $K$, and $r\geq1$. The functional
\begin{align}\label{santalofunctional}
     W_{r}^{s\circ}(K;N;C)
    &=\nonumber
    \frac{1}{|K|^{N+r}}\displaystyle\int_{K^N} |([x_1\cdots x_N]C)^{s\circ}|^{-r} \, dx_1\cdots dx_N\\[0.5em]
    &=
    \dfrac{1}{\vol{K}^{r}}\E\vol{\left([X_1\cdots X_N]C\right)^{s\circ}}^{-r},
\end{align}
expresses the normalized higher negative moments of the volume, of the polar of a non-symmetric random polytope in $K$ with respect to its Santal\'o point.

The following theorem by Meyer and Reisner \cite{meyreishadow}, allows us to show for \eqref{santalofunctional} an analogous result to Theorem~\ref{maintheorem}. We recall a non-degenerate shadow system, $K_t$, is a shadow system with non-empty interior for all $t$ in the interval.

\begin{Th}\label{meyreith}
    Let $K_t$, $t\in[a,b]$, be a non-degenerated shadow system in $\R^n$. Then $\vol{K_t^{s\circ}}^{-1}$ is a convex function of $t$.
\end{Th}

As in the symmetric case, \eqref{santalofunctional} is finite, continuous with respect to the Hausdorff metric, and invariant under invertible linear transformations and translations; the Santal\'{o} point optimizes the translation of all points in the interior of $K$. In order to use Theorem~\ref{meyreith}, we consider  $C=\mathop{\rm conv}\{e_1,\ldots,e_N\}$.

\begin{Th}\label{santalotheo}
    Let $C=\mathop{\rm conv}\{e_1,\ldots,e_N\}$, $K_t$ an RS-Movement for $t\in[-1,1]$, $r\geq 1$, and $N\geq n+1$. Then
    	\[
    	t\longmapsto W^{s\circ}_{r}(K_t;N;C)
    	\]
    is a convex function of $t$.
\end{Th}

\begin{proof}
    Suppose without loss of generality $K_0=K$, so $PK=PK_t$ and $\vol{K}=\vol{K_t}$ for all $t\in[-1,1]$. Let $u$ and $\ell$ be as in \S~\ref{Prelim}. Then, as in Theorem~\ref{maintheorem}, the convexity of \eqref{santalofunctional} follows by Fubini's Theorem and Theorem \ref{meyreith}.
\end{proof}

\section{Applications}

Here we present some applications of Theorem \ref{maintheorem} and Theorem \ref{santalotheo}. In particular, we argue the maximizers for our functionals \eqref{mainfunctional} and \eqref{santalofunctional} are squares and triangles, respectively, and prove by a random approximation procedure important results from Convex Geometry in the plane where a local stochastic dominance holds. We start by showing the maximizers for \eqref{mainfunctional}.

    \begin{Lemma}\label{quadrilateralsmaximize}
    Let $N\geq2$, $C\in\mathcal{K}_{0}^N$ origin symmetric, and $r\geq 1$. Then
        \begin{equation*}
            W_r(K;N;C) \leq W_r(B_{\infty}^2;N;C)
        \end{equation*}
    for all centrally symmetric $K\in\mathcal{K}_{0}^2$.
    \end{Lemma}

    \begin{proof}
        Let $m\geq3$, $K$ a polygon with $\{\pm a_1,\ldots,\pm a_m\}$ ordered clockwise vertices, $\theta$ a direction parallel to the line joining $a_1$ and $a_3$, and set
            \[
            K_t := {\rm conv}\{\pm a_1, a_2+t\theta, \pm a_3, \ldots,\pm a_m\}.
            \]
        Then, there exists $\delta_1,\delta_2>0$ such that for $t\in[-\delta_1,\delta_2]$, $K_t$ is an RS-movement where $K_{-\delta_1}$ and $K_{\delta_2}$ have $m-2$ vertices. As $W_r(K;N;C)$ is a convex functional, Theorem \ref{maintheorem},
            \[
            W_r(K;N;C) \leq \max\{ W_r(K_{-\delta_1};N;C), W_r(K_{\delta_2};N;C) \}
            \]
        and iterating the above procedure one has
            \[
            W_r(K;N;C) \leq W_r(P;N;C),
            \]
        where $P$ is a parallelogram. The result follows by the invariance under linear transformations and the continuity of the functional.
    \end{proof}

We recall now a convergence result that will allow us to recover the classical inequalities, for the proof we refer to \cite{Schneider-Book}.
\begin{Lemma}
\label{Hausdorff-Convergence}
    Let $K, K_1, \ldots, K_N$ be convex bodies with the origin as an interior point with $K_N \overset{\delta^H}{\longrightarrow} K$ as $N\to\infty$. Then
        \[
            K_N^{\circ} \overset{\delta^H}{\longrightarrow} K^{\circ} \text{ as } N\to\infty.
        \]
\end{Lemma}

Let $T$ be a triangle of volume one and centroid the origin, and  $B_p^N$ denote the unit ball in $\ell_p^N$, for $1\leq p\leq \infty$. Let $(\Omega,\mathcal{F},\Pro)$ be a probability space, and assume we have the following independent random vectors sampled uniformly according to normilized the Lebesgue measure on the given set:
    \begin{itemize}
    \setlength\itemsep{.5em}
        \item
            $\{X_i\}_{i=1}^N$ sampled in $K$;
        \item
            $\{{Y}_i\}_{i=1}^N$ sampled in $B_{\infty}^2$;
        \item
            $\{\widetilde{Y}_i\}_{i=1}^N$ sampled in $T$.
    \end{itemize}
Let $[\mathbf{X}]$, $[\mathbf{Y}]$, and $[\mathbf{\widetilde{Y}}]$ as in \S\ref{Prelim}. We now recall that under almost sure convergence, e.g. \cite{CFPP}, integration and limit operation can be interchanged, so by Lemma~\ref{Hausdorff-Convergence} one has the following almost surely convergence in $\delta^H$:

    \begin{center}
        \begin{minipage}[t]{0.5\textwidth}
        \centering
            \begin{itemize}
            \setlength\itemsep{1em}
                \item $[\mathbf{X}]B_1^N \to K$.
                \item $\dfrac{1}{N^{1/p}}[\mathbf{X}]B_q^N \to Z_pK$.
                \item $\dfrac{1}{N}[\mathbf{X}]B_{\infty}^N \to ZK$.
            \end{itemize}
        \end{minipage}%
        \begin{minipage}[t]{0.5\textwidth}
        \centering
            \begin{itemize}
            \setlength\itemsep{1em}
                \item $[\mathbf{Y}]B_1^N \to B_{\infty}^2$.
                \item $\dfrac{1}{N^{1/p}}[\mathbf{{Y}}]B_q^N \to Z_p B_{\infty}^2$.
                \item $\dfrac{1}{N}[\mathbf{{Y}}]B_{\infty}^N \to Z B_{\infty}^2$.
            \end{itemize}
        \end{minipage}
    \end{center}

Now, using Theorem~\ref{maintheorem} and Lemma~\ref{quadrilateralsmaximize}, we deduce the proofs of the stochastic forms on the plane for Mahler's Theorem \cite{Mah:1939-Minimal}, Theorem~\ref{Stochastic-Symmetric-Mahler}, and a reverse Lutwak-Zhang's inequality \cite{LutZha-1997-Blaschke}, Theorem~\ref{Stochastic-Reverse-Lutwak-Zhang}.

\begin{proof}[Proof of Theorem~\ref{Stochastic-Symmetric-Mahler}]
    Consider without loss of generality $K$ such that $\vol{K}=\vol{B_{\infty}^2}$ and notice we are able to write
        \[
            [K]_N=[\mathbf{X}]B_1^N
            \text{\hspace{1em} and \hspace{1em}}
            [B_{\infty}^2]_N=[\mathbf{Y}]B_1^N.
        \]
    By Lemma~\ref{quadrilateralsmaximize} with $C=B_1^N$ we have
        \[
            \frac{1}{\vol{K}^{r}}\E\vol{([\mathbf{X}]B_1^N)^{\circ}}^{-r} \leq \frac{1}{\vol{B_{\infty}^2}^{r}}\E\vol{([\mathbf{Y}]B_1^N)^{\circ}}^{-r}
        \]
    and inequality \eqref{Stochastic-Mahler} follows.
\end{proof}

\begin{Rem}
    From \eqref{Stochastic-Mahler} one has
        \[
            \lim\limits_{N\to\infty}\E\vol{([\mathbf{X}]B_1^N)^{\circ}}^{-r}
            \leq
            \lim\limits_{N\to\infty}\E\vol{([\mathbf{Y}]B_1^N)^{\circ}}^{-r}.
        \]
    Thus, by the continuity of the Lebesgue measure and Dominated Convergence Theorem
        \[
            \E\vol{(\lim_{N\rightarrow\infty}[\mathbf{X}]B_1^N)^\circ}^{-r}
            \leq
            \E\vol{(\lim_{N\rightarrow\infty}[\mathbf{Y}]B_1^N)^\circ}^{-r}.
        \]
    Therefore by Lemma~\ref{Hausdorff-Convergence} one has the almost sure convergence mentioned above. It follows that
        \[
            \vol{K^{\circ}}^r \geq \vol{(B_{\infty}^2)^{\circ}}^r = \vol{B_1^2}^r.
        \]
    In particular, we recover \eqref{Mahler-Conjecture} on the plane, i.e.,
        \[
            \vol{K}\vol{K^{\circ}} \geq \vol{B_{\infty}^2}\vol{B_1^2}.
        \]
\end{Rem}

\begin{proof}[Proof of Theorem~\ref{Stochastic-Reverse-Lutwak-Zhang}]
    First consider without loss of generality $K$ such that $\vol{K}=\vol{B_{\infty}^2}$. Notice for $1/p+1/q=1$, \eqref{Empirical-p-Centroid-Body} can be compare in matrix form with
        \[
            Z_{p,N}(K) = \frac{1}{N^{1/p}}[\mathbf{X}]B_q^N
            \text{\hspace{1em} and \hspace{1em}}
            Z_{p,N}(B_{\infty}^2) = \frac{1}{N^{1/p}}[\mathbf{{Y}}]B_q^N.
        \]
    Then, by Lemma~\ref{quadrilateralsmaximize} with $C=B_q^N$ one has
        \[
            \frac{1}{\vol{K}^{r}}\E\vol{([\mathbf{X}]B_q^N)^{\circ}}^{-r}
            \leq \frac{1}{\vol{B_{\infty}^2}^{r}}\E\vol{([\mathbf{\widetilde{Y}}]B_q^N)^{\circ}}^{-r},
        \]
    so it follows that
        \[
            \E\vol{\parenth{\frac{1}{N^{1/p}}[\mathbf{X}]B_q^N}^{\circ}}^{-r}
            \leq
            \E\vol{\parenth{\frac{1}{N^{1/p}}[\mathbf{{Y}}]B_q^N}^{\circ}}^{-r}
        \]
    which is \eqref{Eq-Stochastic-Reverse-Lutwak-Zhang}. Moreover, taking limits on both sides of the expression above
        \[
            \lim\limits_{N\to\infty}\E\vol{\parenth{\frac{1}{N^{1/p}}[\mathbf{X}]B_q^N}^{\circ}}^{-r} \leq \lim\limits_{N\to\infty}\E\vol{\parenth{\frac{1}{N^{1/p}}[\mathbf{{Y}}]B_q^N}^{\circ}}^{-r}.
        \]
    Inequality \eqref{Reverse-Lutwak-Zhang} follows from the continuity of the Lebesgue measure, the double application of the Dominated Convergence Theorem, and the almost sure convergence to the $p$-centroid body from the beginning of this section.
\end{proof}

\begin{Rem}

Using \eqref{p-centroid-body} for $p=1$ and \eqref{polarvolume}, one can determine the volume of the polar centroid body of $B_{\infty}^2$
    \begin{align}
    \label{squarecentroidpolar}
    \nonumber
        \vol{(ZB_{\infty}^2)^{\circ}}
        &=
        8\displaystyle\int_S\left(\int_{B_{\infty}^2} \vol{\ang{x,z}} dx\right)^{-2}dz\\
    \nonumber
        &=
        32\int_0^{\pi/2}{\parenth{\int_{-1}^1{\int_{-1}^1{\abs{x\cos{\theta}+y\sin{\theta}}}dx}dy}^{-2}}d\theta\\
        &=
        \dfrac{4\pi}{\sqrt{3}} + 6.
    \end{align}
Therefore, using \eqref{squarecentroidpolar} and \eqref{Polar-Bis-Bor} we can give an estimate for the planar Centroid volume product:
        \[
            \vol{(ZK)^{\circ}}\vol{K} \geq \frac{16\pi}{\sqrt{3}} + 24.
        \]
\end{Rem}

Lastly, to finish the section we first show in Lemma~\ref{trianglesextremizers} that triangles are the maximizers for functional \eqref{santalofunctional} which allows us to give a stochastic Mahler's Theorem on the plane for non-symmetric convex bodies.

\begin{Lemma}\label{trianglesextremizers}
    Let $T$ be a triangle, $N\geq3$, $C={\rm conv}\{e_1,\ldots,e_N\}$, and $r\geq 1$. Then
        \begin{equation*}
            W_{r}^{s\circ}(K;N;C) \leq W_{r}^{s\circ}(T;N;C)
        \end{equation*}
    for all $K\in\mathcal{K}_0^N$.
\end{Lemma}

\begin{proof}
    Let $m\geq3$, $K$ a polygon with $\{a_1, \ldots, a_m\}$ ordered clockwise vertices, $\theta$ a direction parallel to the line joining $a_1$ and $a_3$, and set
        \[
            K_t := {\rm conv}\{a_1, a_2+t\theta, a_3, \ldots, a_m\}.
        \]
    Then, there exists $\delta_1,\delta_2>0$ such that for $t\in[-\delta_1,\delta_2]$, $K_t$ is an RS-movement where $K_{-\delta_1}$ and $K_{\delta_2}$ have $m-1$ vertices. As $W^{s\circ}_{p,r}(K;N;C)$ is a convex functional, Theorem~\ref{santalotheo},
        \[
            W_{r}^{s\circ}(K;N;C) \leq \max\{ W_{r}^{s\circ}(K_{-\delta_1};N;C), W_{r}^{s\circ}(K_{\delta_2};N;C)\}
        \]
    and iterating the above procedure one has
       \[
        W_{r}^{s\circ}(K;N;C) \leq W_{r}^{s\circ}(T;N;C).
       \]
    The result follows by the invariance under linear transformations and the continuity of the functional.
\end{proof}

\begin{proof}[Proof of Theorem~\ref{Stochastic-Non-Symmetric-Mahler}]
    Let $e_1,\dots,e_N$ be the standard orthonormal basis in $\R^n$, and $\tilde{C}={\rm conv}\{e_1,\dots,e_N\}$. Consider without loss of generality $K$ such that $\vol{K}=\vol{T}=1$, then we are able to write
        \[
            [K]_N=[\mathbf{X}]\tilde{C}
            \text{\hspace{1em} and \hspace{1em}}
            [T]_N=[\mathbf{\widetilde{Y}}]\tilde{C}.
        \]
    By Lemma~\ref{trianglesextremizers} with $C=\tilde{C}$ we have
        \[
            \frac{1}{\vol{K}^{r}}\E\vol{([\mathbf{X}]\tilde{C})^{s\circ}}^{-r}
            \leq
            \frac{1}{\vol{T}^{r}}\E\vol{([\mathbf{\widetilde{Y}}]\tilde{C})^{s\circ}}^{-r}
        \]
    and inequality \eqref{Stochastic-Mahler-Santalo} follows.
\end{proof}

\begin{Rem}
    From \eqref{Stochastic-Mahler-Santalo} one has
        \[
            \lim\limits_{N\to\infty}\E\vol{([\mathbf{X}]\tilde{C})^{s\circ}}^{-r}
            \leq
            \lim\limits_{N\to\infty}\E\vol{([\mathbf{\widetilde{Y}}]\tilde{C})^{s\circ}}^{-r}.
        \]
    Thus, by the continuity of the Lebesgue measure and Dominated Convergence Theorem
        \[
            \E\vol{(\lim_{N\rightarrow\infty}[\mathbf{X}]\tilde{C})^{s\circ}}^{-r}
            \leq
            \E\vol{(\lim_{N\rightarrow\infty}[\mathbf{\widetilde{Y}}]\tilde{C})^{s\circ}}^{-r}.
        \]
    Therefore by Lemma~\ref{Hausdorff-Convergence} one has the almost sure convergence mentioned above. It follows that
        \[
            \vol{K^{s\circ}}
            \geq
            \vol{T^{s\circ}}
            =
            \frac{27}{4},
        \]
    which recovers \eqref{Mahler-Conjecture}.
\end{Rem}

\noindent {{\bf Acknowledgments} I would like to express my gratitude to my advisor, Peter Pivovarov, for his valuable feedback. His technical and editorial advice was essential to the completion of this paper. I am also grateful to the anonymous referee for providing insightful comments, detailed suggestions, and corrections leading to an improvement of the paper.}

%
% BIBLIOGRAPHY
%

\addcontentsline{toc}{section}{References}
\bibliographystyle{plain}
\bibliography{biblio}

\end{document}